\documentclass[12pt]{amsart}

\usepackage{amssymb}
\usepackage{bm}
\usepackage{graphicx}
\usepackage{psfrag}
\usepackage[usenames,dvipsnames]{xcolor}
\usepackage{enumerate}
\usepackage{multirow}
\usepackage{url}

\usepackage[backref]{hyperref}

\usepackage{comment}

\usepackage{algpseudocode}

\usepackage{mathtools}

\usepackage{adjustbox}
\usepackage{blkarray}
\usepackage{tikz-cd}

\usetikzlibrary{decorations.pathreplacing,calligraphy}
\usetikzlibrary{decorations.pathmorphing}
\usetikzlibrary{shapes.multipart}

\newcommand{\excise}[1]{}%{$\star$\textsc{#1}$\star$}

\newtheorem{thm}{Theorem}[section]
\newtheorem{lemma}[thm]{Lemma}

\newtheorem{cor}[thm]{Corollary}
\newtheorem{prop}[thm]{Proposition}
\newtheorem{conj}[thm]{Conjecture}

\theoremstyle{definition}

\newtheorem{notation}[thm]{Notation}

\numberwithin{equation}{section}

\renewcommand\>{\rangle}
\newcommand\<{\langle}

\newcommand\ZZ{\mathbb{Z}}

\newcommand\til{\mathord\sim}

\DeclareMathOperator\Betti{Betti} % Betti
\DeclareMathOperator\lcm{lcm} % lcm
\DeclareMathOperator\Ap{Ap} % Apery set
\DeclareMathOperator\supp{supp} % support
\tikzcdset{scale cd/.style={every label/.append style={scale=#1},cells={nodes={scale=#1}}}}
\tikzset{snake it/.style={decorate, decoration=snake}}

\begin{document}%%%%%%%%%%%%%%%%%%%%%%%%%%%%%%%%%%%%%%%%%%%%%%%%%%%%%%%%
%%%%%%%%%%%%%%%%%%%%%%%%%%%%%%%%%%%%%%%%%%%%%%%%%%%%%%%%%%%%%%%%%%%%%%%%

\mbox{}
%\vspace{-2ex}%-1.1743pt}
\title[On numerical semigroups and $\ell_0$- and $\ell_\infty$-norms of factorizations]{On numerical semigroup elements and \\ the $\ell_0$- and $\ell_\infty$-norms of their factorizations}

\author[S.~Cyrusian]{Sogol Cyrusian}
\address{Mathematics Department\\University of California Santa Barbra\\Santa Barbara, CA 93106}
\email{sogol@ucsb.edu}

\author[A.~Domat]{Alex Domat}
\address{Mathematics Department\\Trinity College\\Hartford, CT 06106}
\email{alexander.domat@trincoll.edu}

\author[C.~O'Neill]{Christopher O'Neill}
\address{Mathematics Department\\San Diego State University\\San Diego, CA 92182}
\email{cdoneill@sdsu.edu}

\author[V.~Ponomarenko]{Vadim Ponomarenko}
\address{Mathematics Department\\San Diego State University\\San Diego, CA 92182}
\email{vponomarenko@sdsu.edu}

\author[E.~Ren]{Eric Ren}
\address{Mathematics Department\\Arizona State University\\Tempe, AZ 85287}
\email{reneric.2002@gmail.com}

\author[M.~Ward]{Mayla Ward}
\address{Mathematics Department\\Western Washington University\\Bellingham, WA 98225}
\email{maylacward@gmail.com}

\makeatletter
\@namedef{subjclassname@2020}{\textup{2020} Mathematics Subject Classification}
\makeatother
% 13D02 Commutative algebra - Syzygies, resolutions, complexes and commutative rings
% 20M14 Group theory and generalizations - Commutative semigroups
% 52B05 Convex and discrete geometry - Combinatorial properties of polytopes and polyhedra (number of faces, shortest paths, etc.)
% 13F65 Commutative rings defined by binomial ideals, toric rings, etc. [See also 14M25]
% 05E40 Combinatorial aspects of commutative algebra
% 13F20 Polynomial rings and ideals; rings of integer-valued polynomials


\date{\today}

\begin{abstract}
A numerical semigroup $S$ is a cofinite, additively-closed subset of $\mathbb Z_{\ge 0}$ that contains 0, and a factorization of $x \in S$ is a $k$-tuple $z = (z_1, \ldots, z_k)$ where $x = z_1a_1 + \cdots + z_ka_k$ expresses $x$ as a sum of generators of $S = \langle a_1, \ldots, a_k \rangle$.  Much~of the study of non-unique factorization centers on factorization length $z_1 + \cdots + z_k$, which coincies with the $\ell_1$-norm of $z$ as the $k$-tuple.  In this paper, we study the $\ell_\infty$-norm and $\ell_0$-norm of factorizations, viewed as alternative notions of length, with particular focus on the generalizations $\Delta_\infty(x)$ and $\Delta_0(x)$ of the delta set $\Delta(x)$ from classical factorization length.  
We prove that the $\infty$-delta set $\Delta_\infty(x)$ is eventually periodic as a function of $x \in S$, classify $\Delta_\infty(S)$ and the 0-delta set $\Delta_0(S)$ for several well-studied families of numerical semigroups, and identify families of numerical semigroups demonstrating $\Delta_\infty(S)$ and $\Delta_0(S)$ can be arbitrarily long intervals and can avoid arbitrarily long subintervals.  
\end{abstract}

\maketitle

% \setcounter{tocdepth}{1}
% \tableofcontents

%%%%%%%%%%%%%%%%%%%%%%%%%%%%%%%%%%%%%%%%%%%%%%%%%%%%%%%%%%%%%%%%%%%%%%%%%
\section{Introduction}%%%%%%%%%%%%%%%%%%%%%%%%%%%%%%%%%%%%%%%%%%%%%%%%%%%
\label{sec:intro}%%%%%%%%%%%%%%%%%%%%%%%%%%%%%%%%%%%%%%%%%%%%%%%%%%%%%%%%
%raggedbottom%%%%%%%%%%%%%%%%%%%%%%%%%%%%%%%%%%%%%%%%%%%%%%%%%%%%%%%%%%%%

A \emph{numerical semigroup} is a cofinite, additively closed set $S \subseteq \ZZ_{\ge 0}$ containing 0.  We~often specify a numerical semigroup via a list of generators, i.e.,
$$S = \<a_1, \ldots, a_k\> = \{z_1a_1 + \cdots + z_ka_k : z_i \in \ZZ_{\ge 0}\}.$$
As~ubiquitous mathematical objects, numerical semigroups arise in countless settings across the mathematics spectrum; see~\cite{numericalappl,numerical} for a thorough introduction.  
Most notably for this manuscript, numerical semigroups arise frequently in factorization theory~\cite{nonuniq} and discrete optimization~\cite{knapsacksurvey}.  

A \emph{factorization} of an element $x \in S$ is an expression
$$
x = z_1a_1 + \cdots + z_ka_k
$$
of $x$ with each $z_i \in \ZZ_{\ge 0}$.  The \emph{support} and \emph{length} of a factorization $z$ are
$$
\supp(z) = \{i : z_i > 0\},
\qquad \text{and} \qquad
\ell_1(z) = z_1 + \cdots + z_k,
$$
respectively.  We denote by
$$
\mathsf Z(x) = \{z \in \ZZ_{\ge 0}^k : x = z_1a_1 + \cdots + z_ka_k\}
\qquad \text{and} \qquad
\mathcal L(x) = \{\ell_1(z) : z \in \mathsf Z(x)\}
$$
the \emph{set of factorizations} and \emph{length set} of $x$, respectively.  Factorization lengths are a cornerstone of factorization theory, and numerous combinatorial invariants derived from length sets used to quantify and compare the non-uniqueness of factorizations across rings and semigroups~\cite{setsoflengthmonthly}.  
One of the more popular such invariants is the \emph{delta set}, which is defined on semigroup elements as
$$
\Delta(x) = \{c_i - c_{i-1} : i = 2, \ldots, r\}
\qquad \text{where} \qquad
\mathcal L(x) = \{c_1 < \cdots < c_r\},
$$
and defined on semigroups as $\Delta(S) = \bigcup_{x \in S} \Delta(x)$.  For numerical semigroups, $\Delta(x)$ is known to be eventually periodic as a function of $x$~\cite{deltaperiodic}, and $\Delta(S)$ is more varied than for some other well-studied families of semigroups~\cite{delta}, such as Krull monoids~\cite{subdeltas}.  

In this paper, we study the \emph{0-length} and \emph{$\infty$-length} of factorizations $z$, which are
$$
\ell_0(z) = |\supp(z)|
\qquad \text{and} \qquad
\ell_\infty(z) = \max(z_1, \ldots, z_k),
$$
respectively.  For each $p \in \{0,1,\infty\}$, we define the \emph{$p$-length set} of $x$ as
$$
\mathcal L_p(x) = \{\ell_p(z) : z \in \mathsf Z(x)\}.  
$$
(when $p = 1$, we recover the classical definitions).  
In discrete optimization, factorizations achieving minimal 0-length are known as \emph{sparse solutions} and have been studied in the context of numerical semigroups~\cite{quasi0norm} as well as for more general semigroups~\cite{bde09,sparsesensingfinitegeometry}.  Additionally, the asymptotic behavior of $\infty$-length was recently studied in~\cite{pnormasymp}, along with the extremal $\ell_p$-norms of factorizations for $p \in [1, \infty) \cap \ZZ$.  

In this paper, we study the \emph{$p$-delta set} of $x$, defined as
$$
\Delta_p(x) = \{c_i - c_{i-1} : i = 2, \ldots, r\}
\qquad \text{where} \qquad
\mathcal L_p(x) = \{c_1 < \cdots < c_r\},
$$
and the \emph{$p$-delta set} of $S$, defined as $\Delta_p(S) = \bigcup_{x \in S} \Delta_p(x)$.  

The contributions of this manuscript are two-fold.  First, we prove several structural results about the set $\mathcal L_\infty(x)$ for large elements $x \in S$.  Our results are reminiscent of the \emph{structure theorem for sets of length}, which drives much of the study of factorization theory~\cite{geroldingerlengthsets,setsoflengthmonthly} and a specialized version of which was recently proven for numerical semigroups~\cite{structurethmns}.  We derive as a consequence that $\Delta_\infty(x)$ is an eventually periodic function of $x \in S$ (Theorem~\ref{t:inftydeltaperiodic}), a result that is also known for the classical delta set~\cite{deltaperiodic} and joins a vast literature of eventual-peridicity results for large numerical semigroup elements~\cite{numericalsurvey}.  

Second, we characterize $\Delta_\infty(S)$ and $\Delta_0(S)$ for several well-studied families of numerical semigroups, and demonstrate via explicit families of numerical semigroups that $\Delta_\infty(S)$ and $\Delta_0(S)$ can each be arbitrarily long intervals and in general can contain arbitrarily long ``gaps''.  
Our results lead us to make the following conjecture.  

\begin{conj}\label{conj:deltarealization}
For every finite set $D \subset \ZZ_{\ge 1}$ with $1 \in D$, there exists numerical semigroups $S$ and $S'$ with $\Delta_0(S) = D$ and $\Delta_\infty(S') = D$.  
\end{conj}

This part of our work is motivated by the \emph{delta set realization problem}~\cite{deltarealizationnumerical}, which makes an analogous claim for the classical delta set $\Delta(S)$.  The delta set realization problem is known to be difficult, in part because proving a given integer lies outside of $\Delta(S)$ necessitates a large amount of control over the factorization structure of $S$; see~\cite{delta} for examples.  Given this, and the technical nature of our arguments in Sections~\ref{sec:inftydeltas} and~\ref{sec:0deltas}, we suspect Conjecture~\ref{conj:deltarealization} to be difficult in general.

%%%%%%%%%%%%%%%%%%%%%%%%%%%%%%%%%%%%%%%%%%%%%%%%%%%%%%%%%%%%%%%%%%%%%%%%%%%%%%%%%%%%%%
\section{A structure theorem for sets of \texorpdfstring{$\infty$}{infinity}-length}%%
\label{sec:structurethm}%%%%%%%%%%%%%%%%%%%%%%%%%%%%%%%%%%%%%%%%%%%%%%%%%%%%%%%%%%%%%%
%raggedbottom%%%%%%%%%%%%%%%%%%%%%%%%%%%%%%%%%%%%%%%%%%%%%%%%%%%%%%%%%%%%%%%%%%%%%%%%%

\begin{notation}
Throughout this paper, $S = \<a_1,\ldots,a_k\>$ denotes a numerical semigroup with minimal generators $a_1 < a_2 < \cdots < a_k$.  
Additionally, throughout this section, 
$$
A = a_1 + \cdots + a_k,
\qquad
g_i = \gcd(\{a_j:i \neq j\}),
\qquad \text{and} \qquad
S_i = \<\tfrac{1}{g_i}a_j : j \ne i\>
$$
for each $i$.  Additionally, for each $i$, fix $a_i' \in \ZZ$ with $a_i' a_i \equiv 1 \bmod g_i$, let
$$
\mathsf Z(x,i) = \{z \in \mathsf Z(x) : z_i = \ell_\infty(z)\}
\qquad \text{and} \qquad
\mathcal L_\infty(x,i) = \{\ell_\infty(z) : z \in \mathsf Z(x,i)\},
$$
and let
$$
L_\infty(x,i) = \max \mathcal L_\infty(x,i)
\qquad \text{and} \qquad
l_\infty(x,i) = \min \mathcal L_\infty(x,i).
$$
\end{notation}

This section contains several structural results concerning the sets $\mathcal L_\infty(x)$, $\mathcal L_\infty(x,i)$, and $\Delta_\infty(x)$ for large $x \in S$.  We briefly outline these results here.  
\begin{itemize}
\item 
We prove in Theorem~\ref{t:atlantis} that each $\mathcal L_\infty(x,i)$ forms what is known as an \emph{almost arithmetic sequence (AAP)} (i.e., an arithmetic sequence with some missing values near either end), a central ingredient to the classical \emph{structure theorem for sets of length}~\cite{geroldingerlengthsets}.  

\item 
We prove that in the AAP description of $\mathcal L_\infty(x,i)$, the ``missing values'' near either end depend only on the equivalence class of $x$ modulo cetain products of the $a_i$'s and $g_i$'s (Theorem~\ref{t:periodicgaps}).  This result is reminiscent of \cite[Theorem~4.2]{structurethmns}, a more detailed version of the structure theorem for sets of length recently proven for numerical semigroups.  

\item 
Proposition~\ref{p:inftydelta} and Theorem~\ref{t:inftydeltaperiodic} are the culmination of these results, collecting the conclusions drawn about $\Delta_\infty(x)$ for large $x$ and $\Delta_\infty(S)$.  

\end{itemize}
The depiction in Figure~\ref{fig:towerdiagram} illustrates how the structure of each $\mathcal L_\infty(x,i)$ for large $x$ contributes to that of $\mathcal L_\infty(x)$ and $\Delta_\infty(x)$.  

Recall that the \emph{Frobenius number} of $S$ is $\mathsf F(S) = \max(\ZZ_{\ge 0} \setminus S)$, and the \emph{Ap\'ery set} of $S$ with respect to a nonzero element $m \in S$ is
$$
\Ap(S;m) = \{n \in S : n - m \notin S\}.
$$
It is known $\Ap(S;m) = \{0, w_1, \ldots, w_{m-1}\}$, where each $w_i \equiv i \bmod m$ is the smallest element of $S$ in its equivalence class modulo $m$.  

\begin{figure}[t!]
\adjustbox{scale = 0.8,center}{
\begin{tikzpicture}[every text node part/.style={align=right}]
    % first tower
    % top section
    \draw (1, 12) to (2, 12) node [black, right] {$L_\infty(x,1)$}; % L(x,1)
    \draw [dashed, snake it] (1.5, 12) to (1.5, 10.5); % L(x,1) vertical
    \draw (1, 10.5) to (2, 10.5) node [black, right] {$L_\infty(x,1)-B_1$}; % L(x,1)-B1
    % middle dots
    \draw [black, fill] (1.5, 10) circle [radius = 0.02]; % L(x,1) dot
    \draw [black, fill] (1.5, 9.5) circle [radius = 0.02]; % L(x,1) dot
    \draw [black, fill] (1.5, 9) circle [radius = 0.02]; % L(x,1) dot
    \draw [black, fill] (1.5, 8.5) circle [radius = 0.02];
    \draw [black, fill] (1.5, 8) circle [radius = 0.02];
    \draw [black, fill] (1.5, 7.5) circle [radius = 0.02];
    % bracket label between dots
    \draw [magenta, decorate, decoration = {brace}] (1.6, 7.5) to (1.6, 7) node [magenta, right, yshift = 6pt, xshift = 1pt] {$g_1$};
    % more dots
    \draw [black, fill] (1.5, 7) circle [radius = 0.02];
    \draw [black, fill] (1.5, 6.5) circle [radius = 0.02];
    \draw [black, fill] (1.5, 6) circle [radius = 0.02];
    \draw [black, fill] (1.5, 5.5) circle [radius = 0.02];
    \draw [black, fill] (1.5, 5) circle [radius = 0.02]; % l(x,1) dot
    \draw [black, fill] (1.5, 4.5) circle [radius = 0.02]; % l(x,1) dot
    \draw [black, fill] (1.5, 4) circle [radius = 0.02]; % l(x,1) dot
    % bottom part
    \draw (1, 3.5) to (2, 3.5) node [black, right] {$l_\infty(x)+a_k$}; % l(x,1) + b1
    \draw [dashed, snake it] (1.5, 3.5) to (1.5, 2.5); % l(x,1) vertical
    \draw (1, 2.5) to (2, 2.5) node [black, right] {$l_\infty(x,1)$}; % l(x,1)
    
    % second tower
    % top part
    \draw (4.5, 9.25) to (5.5, 9.25) node [black, right] {$L_\infty(x,2)$}; % L(x,2)
    \draw [dashed, snake it] (5, 9.25) to (5, 8.25); % L(x,2) vertical
    \draw (4.5, 8.25) to (5.5, 8.25) node [black, right] {$L_\infty(x,2)-B_2$}; % L(x,2)-B2
    % middle dots
    \draw [black, fill] (5, 7.75) circle [radius = 0.02]; % L(x,2) dot
    \draw [black, fill] (5, 7.25) circle [radius = 0.02]; % L(x,2) dot
    \draw [black, fill] (5, 6.75) circle [radius = 0.02]; % L(x,2) dot
    % bracket label between dots
    \draw [magenta, decorate, decoration = {brace}] (5.1, 6.75) to (5.1, 6.25) node [magenta, right, yshift = 6pt, xshift = 1pt] {$g_2$};
    % more dots
    \draw [black, fill] (5, 6.25) circle [radius = 0.02];
    \draw [black, fill] (5, 5.75) circle [radius = 0.02];
    \draw [black, fill] (5, 5.25) circle [radius = 0.02];
    \draw [black, fill] (5, 4.75) circle [radius = 0.02];
    \draw [black, fill] (5, 4.25) circle [radius = 0.02]; % L(x,2) dot
    \draw [black, fill] (5, 3.75) circle [radius = 0.02]; % l(x,2) dot
    % \draw [black, fill] (5, 3.25) circle [radius = 0.02]; % l(x,2) dot
    % bottom part
    \draw (4.5, 3.5) to (5.5, 3.5) node [black, right] {$l_\infty(x)+a_k$}; % l(x,2) +b2
    \draw [dashed, snake it] (5, 3.5) to (5, 2); % l(x,2) vertical
    \draw (4.5, 2) to (5.5, 2) node [black, right] {$l_\infty(x,2)$}; % l(x,2)
    
    % left labels
    % horizontal magenta lines
    \draw [magenta] (0.5, 12) to (0.9, 12);
    \draw [magenta] (0.5, 10.5) to (0.9, 10.5);
    \draw [magenta] (0.5, 9.5) to (1.4, 9.5);
    \draw [magenta] (0.5, 9.25) to (4.4, 9.25);
    \draw [magenta] (0.5, 8.25) to (4.4, 8.25);
    \draw [magenta] (0.5, 3.5) to (0.9, 3.5);
    % left braces/labels
    \draw [magenta, decorate, decoration = {brace, amplitude=5pt}] (0.5, 10.5) to (0.5, 12) node [magenta, left, yshift = -22pt, xshift = -5pt] {Region (iii) in\\Proposition~\ref{p:inftydelta}};
    \draw [magenta, decorate, decoration = {brace, amplitude=5pt}] (0.5, 9.5) to (0.5, 10.5) node [magenta, left, yshift = -12pt, xshift = -5pt] {$g_1$ Spacing};
    \draw [magenta, decorate, decoration = {brace, amplitude=5pt}] (0.5, 8.25) to (0.5, 9.25) node [magenta, left, yshift = -14pt, xshift = -5pt] {Region (ii) in\\Proposition~\ref{p:inftydelta}};
    \draw [magenta, decorate, decoration = {brace, amplitude=5pt}] (0.5, 3.5) to (0.5, 8.25) node [magenta, left, yshift = -66pt, xshift = -5pt] {$[1,\min(g_1, g_2)]$\\ Spacing};
    \draw [magenta, decorate, decoration = {brace, amplitude=5pt}] (0.5, 1.5) to (0.5, 3.5) node [magenta, left, yshift = -29pt, xshift = -5pt] {Region (i) in\\Proposition~\ref{p:inftydelta}};
    % right most brace/label
    \draw [magenta, decorate, decoration = {brace, amplitude = 5pt}] (8.5, 8.25) to (8.5, 6.6) node [magenta, right, xshift = 5pt, yshift = 22pt] {For all $x \gg 0$, gaps of every\\size in $[1,\min(g_1,g_2)]$ occur\,};

    % subsequent towers
    % third tower
    \draw (8, 6.5) to (9, 6.5) node [black, right] {$L_\infty(x,3)$}; % L(x,3)
    \draw [dashed, snake it] (8.5, 6.5) to (8.5, 1.5);
    % dotted line between tower tops
    \draw [very thick, loosely dotted] (10, 6) to (11, 5.5);
    % last tower
    \draw (11.5, 5.25) to (12.5, 5.25) node [black, right] {$L_\infty(x,k)$}; % L(x,k)
    \draw [dashed, snake it] (12, 5.25) to (12, 1.5);

    % bottom line
    \draw (0.5, 1.5) to (13, 1.5) node [black, right] {$l_\infty(x)$}; % l(x)
\end{tikzpicture}
}
\caption{Diagram of $\infty$-length set elements for large $x \in S$, where the $B_i$ are defined in Theorem~\ref{t:atlantis}}
\label{fig:towerdiagram}
\end{figure}
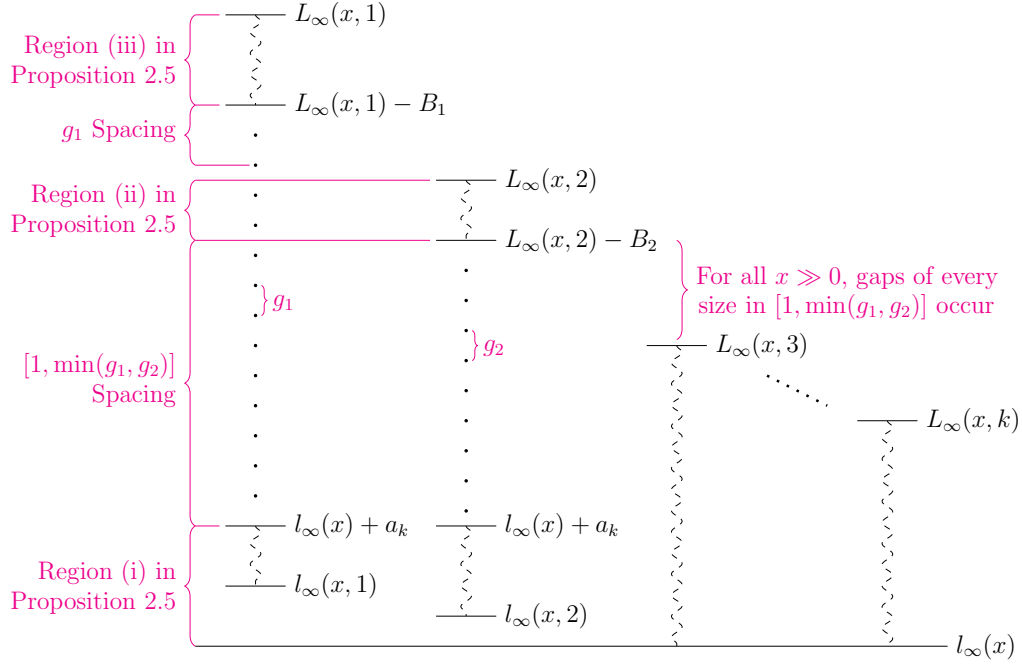

\begin{lemma}\label{l:mininftybounds}
For ever $x \in S$, the following inequalities hold:
\begin{enumerate}[(a)]
\item 
$\tfrac{1}{A}x \le l_\infty(x) \le \tfrac{1}{A}x + a_k$; and

\item 
for each $i$, $\tfrac{1}{a_i}x - ka_k \le L_\infty(x,i) \le \tfrac{1}{a_i}x$.  

\end{enumerate}
\end{lemma}

\begin{proof}
Letting $z \in \mathsf Z(x)$ with $\ell_\infty(z) = l_\infty(x)$, we see
$$
x = z_1a_1 + \cdots + z_ka_k
\le \ell_\infty(z) a_1 + \cdots + \ell_\infty(z) a_k
= l_\infty(x) A.
$$

Next, write $x = a + qA$ for $a \in \Ap(S;A)$.  We claim $l_\infty(a) \le a_k$.  Indeed, by way of contradiction, fix a factorization $z \in \mathsf Z(a)$ with $\ell_\infty(z) = l_\infty(a)$, and assume some $z_j > a_k$.  Some $z_i' = 0$ since $a \in \Ap(S;A)$, so trading $a_i$ copies of $a_j$ for $a_j$ copies of $a_i$ yields a factorization $z' \in \mathsf Z(a)$ with strictly fewer copies of $a_j$ and no new coordinates larger than $a_k$.  After applying such a trade to each maximal entry in $z$, we obtain a factorization $z'$ with $\ell_\infty(z') < \ell_\infty(z) = l_\infty(a)$, which is a contradiction.  

Letting $z \in \mathsf Z(a)$ with $\ell_\infty(z) = l_\infty(a)$, \cite[Theorem~2.6]{pnormasymp} implies $z'' = (z_1 + q, \ldots, z_k + q)$ is a factorization of $x = a + qA$ with $l_\infty(x) = \ell_\infty(z'') \le q + a_k$.  Thus,
$$l_\infty(x) \le q + a_k = \tfrac{1}{A}(x - a) + a_k \le \tfrac{1}{A} x + a_k,$$

Proceeding to part~(b), suppose $z \in \mathsf Z(x)$ satisfies $\ell_\infty(z) = z_i = L_\infty(x,i)$.  Then 
$$
L_\infty(x,i)a_i
\le L_\infty(x,i) + \sum_{j \ne i} z_ja_j
= x.
$$
Additionally, we must have $z_j < a_i$ for $j \ne i$, as otherwise one could trade $a_i$ copies of $a_j$ for $a_j$ copies of $a_i$ and constradict the maximality of $z_i$.  As such, 
$$
x \le L_\infty(x,i)a_i + \sum_{j \ne i} z_ja_j
< L_\infty(x,i)a_i + \sum_{j \ne i} a_ia_k
\le L_\infty(x,i)a_i + ka_ia_k
$$
from which the last remaining inequality is immediately obtained.  
\end{proof}

\begin{thm}\label{t:atlantis}
For each $i = 1, \ldots, k$, there exists $B_i \in \ZZ$ such that
$$
[\tfrac{1}{A}x + a_k, \tfrac{1}{a_i}x - B_i] \cap (g_i\ZZ + a_i'x) \subseteq \mathcal L_\infty(x,i) \subseteq g_i\ZZ + a_i' x.
$$
\end{thm}

\begin{proof}
The second containment holds since $\ell \in \mathcal L_\infty(x,i)$ implies $x - \ell a_i \in g_iS_i$.  
Define
$$
B_i = \tfrac{1}{a_i}g_i(F(S_i) + 1)
$$
and fix $\ell \in g_i\ZZ + a_i' x$ with $\tfrac{1}{A}x + a_k \le \ell \le \tfrac{1}{a_i}x - B_i$.  Since 
$$
x - \ell a_i
= a_i (\tfrac{1}{a_i}x - \ell)
\ge g_i(F(S_i) + 1),
$$
we have $x - \ell a_i \in g_iS_i$.  Moreover, we claim $x - \ell a_i$ has a factorization in $g_iS_i$ of $\infty$-length at most $\ell$.  
Indeed, notice that
$$
x - \ell a_i + a_k(A - a_i)
\le x + a_kA - \ell a_i
= (\tfrac{1}{A}x + a_k)A - \ell a_i
\le \ell (A - a_i),
$$
from which we obtain
$$
\tfrac{1}{A - a_i}(x - \ell a_i) + \max(\{a_j : j \ne i\})
\le \tfrac{1}{A - a_i}(x - \ell a_i) + a_k
\le \ell.
$$
Now applying Lemma~\ref{l:mininftybounds} to $g_iS_i$ implies $x - \ell a_i$ has a factorization in $g_iS_i$ of $\infty$-length at most $\ell$, which completes the proof.   
\end{proof}

\begin{thm}\label{t:periodicgaps}
Fix $B, B' > 0$.  
For all $x \gg 0$, we have
$$
\mathcal L_\infty(x + a_i, i) \cap [\tfrac{1}{a_i}(x + a_i) - B, \infty) = 1 + \big( \mathcal L_\infty(x,i) \cap [\tfrac{1}{a_i}x - B, \infty) \big)
$$
for each $i$, as well as
$$
\mathcal L_\infty(x + A) \cap [0, \tfrac{1}{A}(x + A) + B'] = 1 + \big( \mathcal L_\infty(x) \cap [0, \tfrac{1}{A}x + B'] \big)
$$
In particular, these hold whenever $x > a_i^2C + a_iB$ and $x > \tfrac{1}{a_i}A(A - a_i)B'$, respectively, where $C = \lceil \tfrac{1}{a_j}(B + 1) \rceil$.  
\end{thm}

\begin{proof}
One can readily check $z \in \mathsf Z(x,i)$ implies $z + e_i \in \mathsf Z(x + a_i, i)$, which shows one containment in the first equality.  For the converse direction, we first claim any factorization $z \in \mathsf Z(x,i)$ with $\ell_\infty(z) = z_i \ge \tfrac{1}{a_i}(x + a_i) - B$ has $z_i > z_j$ for all $j \ne i$.  Indeed, if $z_j = z_i$ for some~$j$, then 
$$
z_j
\ge \tfrac{1}{a_i}x - B
> a_iC + B - B
\ge a_iC,
$$
so trading $a_iC$ copies of $a_j$ in $z$ for $a_jC$ copies of $a_i$ yields a factorization in $\mathsf Z(x,i)$ with $i$-th coordinate
$$
z_i + a_jC
= z_i + a_j\lceil \tfrac{1}{a_j}(B + 1) \rceil
> \tfrac{1}{a_i}x - B + B
= \tfrac{1}{a_i}x
$$
which is impossible by Lemma~\ref{l:mininftybounds}(b).  
Having now proven the claim, any $z' \in \mathsf Z(x + a_i,i)$ has $z' - e_i \in \mathsf Z(x,i)$, and the first equality is proven.  

Analogously for the second equality, any $z \in \mathsf Z(x)$ has $z' = z + e_1 + \cdots + e_k \in \mathsf Z(x + A)$.  For the reverse containment, we claim any $z \in \mathsf Z(x)$ with $\ell_\infty(z) \le \tfrac{1}{A}x + B'$ has no zero entries.  Indeed, if some $z_i = 0$, then 
\begin{align*}
x
&= \sum_{j \ne i} z_j a_j
\le \sum_{j \ne i} (\tfrac{1}{A}x + B')a_j
= (\tfrac{1}{A}x + B')(A - a_i)
= x - \tfrac{1}{A}a_i x + (A - a_i)B'
\\
&< x - (A - a_i)B' + (A - a_i)B'
= x
\end{align*}
a contradiction.  As such, any factorization $z' \in \mathsf Z(x + A)$ with $\ell_\infty(z') \le \tfrac{1}{A}x + B' + 1$ has $z' - e_1 - \cdots - e_k \in \mathsf Z(x)$, thereby completing the proof.  
\end{proof}

\begin{prop}\label{p:inftydelta}
For all $x \gg 0$, we have $[1, \min(g_1,g_2)] \cup \{g_1\} \subseteq \Delta_\infty(x)$.  Moreover, if $\ell < \ell'$ are successive elements of $\mathcal L_\infty(x)$ with $\ell' - \ell \notin [1, \min(g_1,g_2)] \cup \{g_1\}$, then at least one of $\ell$ and $\ell'$ lies in one of the following intervals:
$$
\text{(i) } [\tfrac{1}{A}x, \tfrac{1}{A}x + a_k];
\qquad
\text{(ii) } [\tfrac{1}{a_2}x - B_2, \tfrac{1}{a_2}x]; 
\qquad \text{or} \qquad
\text{(iii) } [\tfrac{1}{a_1}x - B_1, \tfrac{1}{a_1}x].
$$
\end{prop}

\begin{proof}
First, if $x$ is large enough that $(\tfrac{1}{a_1}x - B_1) - \tfrac{1}{a_2}x > 3g_1$, then by Lemma~\ref{l:mininftybounds}(b), 
$$\mathcal L_\infty(x) \cap (\tfrac{1}{a_2}x, \tfrac{1}{a_1}x - B_1) = \mathcal L_\infty(x,1) \cap (\tfrac{1}{a_2}x, \tfrac{1}{a_1}x - B_1)$$
contains at least 2 lengths, and any two consecutive lengths therein must have difference $g_1 \in \Delta_\infty(x)$ by Theorem~\ref{t:atlantis}.  

Analogously, if $x$ is large enough that $(\tfrac{1}{a_2}x - B_2) - \tfrac{1}{a_3}x > 2g_1g_2$, then by Lemma~\ref{l:mininftybounds}(b) and Theorem~\ref{t:atlantis},
\begin{align*}
\mathcal L_\infty(x) \cap (\tfrac{1}{a_3}x, \tfrac{1}{a_2}x - B_2)
&= \big( \mathcal L_\infty(x,1) \cup \mathcal L_\infty(x,2) \big) \cap (\tfrac{1}{a_3}x, \tfrac{1}{a_2}x - B_2)
\\
&= \big( (g_1\ZZ + a_1' x) \cup (g_2\ZZ + a_2' x) \big) \cap (\tfrac{1}{a_3}x, \tfrac{1}{a_2}x - B_2),
\end{align*}
within which successive elements achieve each difference in $[1, \min(g_1,g_2)]$ by the Chinese Remainder Theorem since $\gcd(g_1,g_2) = \gcd(a_1, \ldots, a_k) = 1$.  

For the final claim, by Theorem~\ref{t:atlantis}, aside from the three claimed intervals, the only subinterval of $[\tfrac{1}{A}x, \tfrac{1}{a_1}x]$ not containing an arithmetic sequences of step size $\min(g_1, g_2)$ is $[\tfrac{1}{a_2}x, \tfrac{1}{a_1}x - B_1]$, whose lengths form an arithmetic sequence of step size $g_1$.  
\end{proof}

\begin{thm}\label{t:inftydeltaperiodic}
For all $x \gg 0$, we have $\Delta_\infty(x + p) = \Delta_\infty(x)$ for $p = \lcm(a_1, g_1a_2, A)$.  
\end{thm}

\begin{proof}
We begin by considering the intervals~(i), (ii), and~(iii) in Proposition~\ref{p:inftydelta}.  Let
\begin{align*}
R_1(x) &= [\tfrac{1}{A}x, \tfrac{1}{A}x + (a_k + g_1)] \cap \mathcal L_\infty(x),
\\
R_2(x) &= [\tfrac{1}{a_2}x - (B_2 + g_1), \tfrac{1}{a_2}x + g_1] \cap \mathcal L_\infty(x),
\\
R_3(x) &= [\tfrac{1}{a_1}x - (B_1 + g_1), \tfrac{1}{a_1}x] \cap \mathcal L_\infty(x).
\end{align*}
Applying Theorem~\ref{t:periodicgaps} with $B' = a_k + g_1$ gives $R_3(x + p) = R_3(x) + \tfrac{1}{A}p$, and letting $B = B_1 + g_1$, we have
$$
R_1(x) = [\tfrac{1}{a_1}x - B, \tfrac{1}{a_1}x] \cap \mathcal L_\infty(x,1)
$$
by Lemma~\ref{l:mininftybounds}(b), so $R_1(x + p) = R_1(x) + \tfrac{1}{a_1}p$.  
Additionally, by Theorem~\ref{t:atlantis}, 
$$
R_2(x) = \big( [\tfrac{1}{a_2}x - B, \tfrac{1}{a_2}x] \cap \mathcal L_\infty(x,2) \big) \cup \big( [\tfrac{1}{a_2}x - B, \tfrac{1}{a_2}x + g_1] \cap (g_1\ZZ + a_1' x) \big)
$$
for $B = B_2 + g_1$.  Clearly $g_1\ZZ + a_1' (x + p) = g_1\ZZ + a_1' x$, and since $g_1 \mid \tfrac{1}{a_2}p$, 
$$\big( g_1\ZZ + a_1' (x + p) \big) + \tfrac{1}{a_2}p = g_1\ZZ + a_1' x$$
as well.  As such, $R_2(x + p) = R_2(x) + \tfrac{1}{a_2}p$ once again by Theorem~\ref{t:periodicgaps}.  

Lastly, by Proposition~\ref{p:inftydelta} any successive lengths in $\mathcal L_\infty(x)$ or $\mathcal L_\infty(x + p)$ not residing in one of the above intervals must have difference in $[1, \min(g_1,g_2)] \cup\{g_1\}$, which is a subset of both $\mathcal L_\infty(x)$ and $\mathcal L_\infty(x + p)$ by Proposition~\ref{p:inftydelta}.  
\end{proof}

\begin{cor}\label{c:machupicchu}
Fix $B > 0$.  For all $x \gg 0$, we have 
$$\ell \in \mathcal L_\infty(x) \cap [\tfrac{1}{a_1}x - B, \tfrac{1}{a_1}x]
\qquad \text{if and only if} \qquad
x - \ell a_1 \in g_1S_1 \cap [0, a_1B].$$
In particular, $\Delta(g_1S_1 \cap (a_1\ZZ + j)) \subseteq \Delta_\infty(S)$ for each $j$.  
\end{cor}

\begin{proof}
As in the proof of Theorem~\ref{t:atlantis}, $\ell \in \mathcal L_\infty(x,1)$ if and only if (i) $x - \ell a_1 \in g_1S_1$ and (ii) there exists a factorization of $x - \ell a_1$ in $g_1S_1$ with $\infty$-length at most $\ell$.  
Tracing through the proof of Theorem~\ref{t:atlantis}, so long as condition~(i) holds, condition~(ii) holds whenever $\ell > \tfrac{1}{A}x + a_k$, which is certainly the case if $\ell > \tfrac{1}{a_1}x - B$ for $x \gg 0$.  Moreover, 
$$x - \ell a_1 < x - (\tfrac{1}{a_1}x - B) a_1 = a_1B.$$
As such, Lemma~\ref{l:mininftybounds}(b) implies the frst claim, and the second claim then follows upon unraveling definitions.  
\end{proof}

%%%%%%%%%%%%%%%%%%%%%%%%%%%%%%%%%%%%%%%%%%%%%%%%%%%%%%%%%%%%%%%%%%%%%%%%%%%
\section{Some families of \texorpdfstring{$\infty$}{infinity}-delta sets}%%
\label{sec:inftydeltas}%%%%%%%%%%%%%%%%%%%%%%%%%%%%%%%%%%%%%%%%%%%%%%%%%%%%
%raggedbottom%%%%%%%%%%%%%%%%%%%%%%%%%%%%%%%%%%%%%%%%%%%%%%%%%%%%%%%%%%%%%%

In this section, we examine the set $\Delta_\infty(S)$ for several families of numerical semigroups.  We characterize the $\infty$-delta set for supersymmetric numerical semigroups~\cite{supersymmetric}, and numerical semigroups whose generators form an arithmetic sequence~\cite{omidalirahmati} or a geometric sequence~\cite{vadimgeomseq,tripathigeomseq}.  We also demonstrate that $\Delta_\infty(S)$ can be an arbitrarily long interval (Theorem~\ref{t:delinfarith}) and have arbitrarily long gaps (Theorem~\ref{t:delinfgaps}).  

Many of our arguments in this section and the next utilize trades, presentations, and Betti elements.  We briefly review the relevant concepts here, though the reader is encouraged to see~\cite[Chapter~5]{numericalappl} and \cite{fingenmon} for a thorough introduction.  

Define an equivalence relation $\til$ on $\ZZ_{\ge 0}^k$ that sets $z \sim z'$ whenever $z, z' \in \mathsf Z(x)$ are factorizations of the same element $x \in S$.  We call each relation $z \sim z'$ between facatorizations of disjoint support a \emph{trade} of $S$, and sometimes identify the difference $z - z' \in \ZZ^k$ with the trade $z \sim z'$.  
A \emph{presentation} of $S$ is a collection $\rho$ of trades with the property that for any $x \in S$ and any $z, z' \in S$, there exists a chain of factorizations 
$$z = y_1 \sim y_2 \sim \cdots \sim y_r = z'$$
wherein $y_i - y_{i-1} \in \rho$ or $y_{i-1} - y_i \in \rho$ for each pair of sequential factorizations $y_{i-1}$ and $y_i$.  
A presentation is \emph{minimal} if it is minimal with respect to containment among all presentations for $S$.  It is known that any two minimal presentations $\rho$ and $\rho'$ for $S$ have the same number of trades, and in fact the set 
$$
\Betti(S) = \{z_1a_1 + \cdots + z_ka_k : z - z' \in \rho \}
$$
of \emph{Betti elements} is independent of the choice of minimal presentation $\rho$.  

\begin{thm}\label{t:delinffamilies}
Let $S = \<a_1, \ldots, a_k\>$ with $a_1 < \cdots < a_k$.  
\begin{enumerate}[(a)]
\item 
If $a < b$ are coprime and each $a_i = a^{k-1-i} b^{i-1}$, then $\Delta_\infty(S) = \{1, 2, \ldots, b\}$.  

\item 
If $p_1, \ldots, p_k \in \ZZ_{\ge 1}$ are pairwise coprime with $p_1 > \cdots > p_k$, $T = p_1 \cdots p_k$, and each $a_i = \tfrac{1}{p_i}T$, then $\Delta_\infty(S) = \{1, 2, \ldots, p_1\}$.  

\end{enumerate}
\end{thm}

\begin{proof}
For part~(a), the trades $be_i \sim ae_{i+1}$ for $i = 1, \ldots, k-1$ form a minimal presentation for $S$ by \cite[Theorem~8]{compseqs}, so $\max \Delta_\infty(S) \le b$.  
Now, if $1 \le c \le a$, we have
$$\mathsf Z((b + a - c)a_1) = \{(b + a - c)e_1, (a - c)e_1 + ae_2\}$$
so $c \in \Delta_\infty(S)$.  Moreover, if $a < c \le b$, then 
$$\mathsf Z(ca_2) = \{(be_1 + (c-a)e_2, ce_2\},$$
so again $c \in \Delta_\infty(S)$.  Thus, $\Delta_\infty(S) = \{1, 2, \ldots, b\}$.  

For part~(b), the trades $p_{i+1}e_i \sim p_ie_{i+1}$ for $i = 1, \ldots, k-1$ form a minimal presentation for $S$ by \cite{supersymmetric}, so $\max \Delta_\infty(S) \le p_1$.  Using a similar argument to part~(a), we have $\Delta_\infty((p_1 + p_2 - c)a_1) = \{c\}$ whenever $1 \le c \le a$ and $\Delta_\infty(ca_2) = \{c\}$ whenever $p_2 < c \le p_1$.  
Thus, $\Delta_\infty(S) = \{1, 2, \ldots, p_1\}$.  
\end{proof}

\begin{thm}\label{t:delinfarith}
Let $S = \<a, a + d, \ldots, a + kd\>$ with $2 \le k < a$ and $\gcd(a,d) = 1$, and write $a - 1 = qk + r$ for $q, r \in \ZZ_{\ge 0}$ with $0 \le r < k$.  Then $\Delta_\infty(S) = \{1, 2, \ldots, q + d + 1\}$.  
\end{thm}

\begin{proof}
In what follows, write $a_i = a + ik$ for $i \in [0,k]$, and for $x \in S$, write factorizations $z \in \mathsf Z(x)$ as $z = (z_0, \ldots, z_k) = z_0e_0 + \cdots + z_ke_k$.  
Before beginning the proof, we recall some facts about arithmetical numerical semigroups; see \cite{setoflengthsets,omidalirahmati}.  Each $x \in S$ has $\ell \in \mathcal L_1(x)$ if and only if 
$$
x = \ell a + bd
\qquad \text{with} \qquad
0 \le b \le k\ell,
$$
as for any factorization $z \in \mathsf Z(x)$ with $\ell_1(z) = \ell$, we can write
$$x = (z_0 + \cdots + z_k)a + (z_1 + 2z_2 + \cdots + kz_k)d$$
with $b = z_1 + 2z_2 + \cdots + kz_k$.  
Moreover,
$$
\ell - \lceil \tfrac{1}{k}b \rceil
\ge z_0
= \ell - (z_1 + \cdots + z_k)
\ge \ell - b,
$$
with equality on the right if $z_2 = \cdots = z_k = 0$.  

We now proceed with the proof.  First, suppose $1 \le G \le d$.  We see
$$
x
= (a+d)a + (d-G)(a + d)
= (a+d-G)(a + d)
$$
are factorizations $z, z' \in \mathsf Z(x)$, respectively, with $\ell_\infty(z) = a+d$ and $\ell_\infty(z') = a+d-G$.  Now, since $\ell_\infty(z') = \ell_1(z') = z_1'$, any factorization $z'' \in \mathsf Z(x)$ with $\ell_\infty(z'') > \ell_\infty(z')$ must have $\ell_1(z'') > \ell_1(z')$.  This means $\ell_1(z'') \ge \ell_1(z)$, and letting $b = z_1'' + 2z_2'' + \cdots + kz_k''$, any such factorization must have 
$$
z_0''
\ge \ell_1(z'') - b
\ge (a + 2d - G) - (d - G)
= a+d
$$
by the first paragraph above.  As such, $G \in \Delta_\infty(x)$.  

Next, suppose $d \le G \le d + q + 1$.  We see
$$
x
= (a+G)a
= (G-d)a + a(a+d)
$$
are factorizations $z, z' \in \mathsf Z(x)$, respectively, with $\ell_\infty(z) = a+G$ and $\ell_\infty(z') = a$.  Fix a factorization $z'' \in \mathsf Z(x)$ with $\ell_\infty(z'') < \ell_\infty(z)$, and let $b = z_1'' + 2z_2'' + \cdots + kz_k''$.  Since $z = (a+G)e_1$, we must have $\ell_1(z'') < \ell_1(z) = a + G$.  As such, $\ell_1(z'') \le a + G - d = \ell_1(z')$ and $b \ge a$, meaning 
$$z_0'' \le \ell_1(z'') - \lceil \tfrac{1}{k}b \rceil \le (a + G - d) - \lceil \tfrac{1}{k}a \rceil \le a + G - d - q - 1 \le a.$$
Additionally, if $z_j'' \ge a$ for some $j \ge 1$, then 
$$\sum_{i \ne j} z_i''a_i = x - z_j'' a_j < x - a a_j \le x - a(a + d) = (G - d)a \le (q + 1)a,$$
and all factorizations of such an element have equal 1-length.  As such, since
$$x = (G - jd)a + a(a + jd),$$
we must have $\ell_1(z'') = \ell_1(z') - jd$, meaning $z''$ coincides with the above factorization.  Thus $\ell_\infty(z'') \le a$, thereby ensuring $G \in \Delta_\infty(x)$.  
\end{proof}

\begin{thm}\label{t:delinfgaps}
Fix $m \ge 3$.  If $S = \<3, 3m + 1, 3m + 2\>$, then 
$$\Delta_\infty(S) = \{1, 2, \ldots , m+1\} \cup \{2m, 2m+1\}.$$
\end{thm}

\begin{proof}
Since $S$ has max embedding dimension (see~\cite[Chapter~3]{numerical}):
\begin{enumerate}[(i)]
\item 
the trades 
$$
(2m + 1)e_1 \sim e_2 + e_3,
\qquad
me_1 + e_3 \sim 2e_2,
\qquad \text{and} \qquad
(m + 1)e_1 + e_2 \sim 2e_3
$$
comprise a minimal presentation for $S$;

\item 
for each $x \in S$, the unique factorization $z = (a,b,c) \in \mathsf Z(x)$ with $\ell_\infty(z)$ maximal in $\mathcal L_\infty(x)$ is also the unique factorization with $b + c \le 1$; and

\item 
for each $a \ge 0$, the factorizations $(a,1,1)$ and $(a+2m-1,0,0)$ of $x = 3(a + 2m - 1)$ have the two highest $\infty$-lengths in $\mathcal L_\infty(x)$.

\end{enumerate}

Fix $x \in S$ and $z = (a,b,c) \in \mathsf Z(x)$.  Suppose $z$ does not have maximal $\infty$-length, and let $G$ be minimal with $\ell_\infty(z) + G \in \mathcal L_\infty(x)$.  By~(iii), if $b = c = 1$, then $G \ge 2m$.  Otherwise, by~(ii) either $b \ge 2$ or $c \ge 2$.  If $b \ge 2$, then fixing $q \in \ZZ$ with $b - 2q \in \{0,1\}$ and performing the second trade in~(i) $q$~times yields a chain of factorizations 
$$(a,b,c) \sim (a + m, b - 2, c + 1) \sim \cdots \sim (a + qm, b - 2q, c + q),$$
wherein each factorization differs in $\infty$-length from the previous factorization by at most $m$, and the final factorization in which has strictly larger $\infty$-length than $z$.  As~such, we have $G \le m$.  By an analogous argument, if $c \ge 2$, then $G \le m + 1$.  This proves $m+2, \ldots, 2m-1 \notin \Delta_\infty(S)$.  

Now, by~(i), we have $\max \Delta_\infty(S) \le 2m + 1$.  We can see by inspection that
$$
\mathsf Z(6m + 3) = \{(2m+1,0,0), (0,1,1)\}
\quad \text{and} \quad
\mathsf Z(6m + 6) = \{(2m+2,0,0), (1,1,1)\},
$$
so $2m, 2m + 1 \in \Delta_\infty(S)$.  Also by inspection,
$$
\mathsf Z(6m + 8) = \{(m+2,0,1), (2,2,0)\}
\quad \text{and} \quad
\mathsf Z(6m + 10) = \{(m+3,1,0), (2,0,2)\},
$$
so $m, m + 1 \in \Delta_\infty(S)$.  
Lastly, for each $G \in \{1, \ldots, m-1\}$, we have
$$
(0,0,G+1), (m+1,1,G-1) \in \mathsf Z(x)
$$
for $x = (G+1)(3m+2)$.  Fix a factorization $z = (a,b,c) \in \mathsf Z(x)$.  Since
$$
3a = x - (3m+1)b + (3m+2)c = (G+1-b-c)(3m+2) + b,
$$
we must have $b + c \le G+1$.  If $b+c = G+1$, then $a = \tfrac{1}{3}b < G + 1$, so $\ell_\infty(z) \le G + 1$.  If $b+c \le G$, then 
$$
3a
\ge (G+1-b-c)(3m+2)
\ge 3m + 2
$$
so $\ell_\infty(z) \ge a \ge m+1$.  This proves $m - G \in \Delta_\infty(S)$.  
\end{proof}

%%%%%%%%%%%%%%%%%%%%%%%%%%%%%%%%%%%%%%%%%%%%%%%%%%%%%%%%%%%%%%%%%%%%%%%%%
\section{Some families of 0-delta sets}%%%%%%%%%%%%%%%%%%%%%%%%%%%%%%%%%%
\label{sec:0deltas}%%%%%%%%%%%%%%%%%%%%%%%%%%%%%%%%%%%%%%%%%%%%%%%%%%%%%%
%raggedbottom%%%%%%%%%%%%%%%%%%%%%%%%%%%%%%%%%%%%%%%%%%%%%%%%%%%%%%%%%%%%

In a similar vein to the prior section, in Theorems~\ref{t:del0families} and~\ref{t:del0gen3} we characterize $\Delta_0(S)$ for numerical semigroups $S$ residing in several well-studied families, including maximal embedding dimension numerical semigroups \cite[Chapter~3]{numerical}, supersymmetric numerical semigroups~\cite{supersymmetric}, 3-generated numerical semigroups~\cite[Chapter~10]{numerical}, and numerical semigroups generated by generalized arithmetic sequences~\cite{omidalirahmati}.  We also identify two families of numerical semigroups achieving notable extremal behavior (Theorems~\ref{t:del0interval} and~\ref{t:del0gaps}).  First, we demonstrate that the structure of $\mathcal L_0(x)$ for large $x \in S$ differs substantially from that of $\mathcal L_\infty(x)$ detailed in Section~\ref{sec:structurethm}.  

\begin{thm}\label{t:del0eventual}
For all $x \gg 0$, we have $\Delta_0(x) = \{1\}$.  In particular, $1 \in \Delta_0(S)$.  
\end{thm}

\begin{proof}
For fixed $I \subseteq \{1, \ldots k\}$ nonempty, and letting $d = \gcd(a_i : i \in I)$, $x \in S$ has a factorization $z \in \mathsf Z(x)$ with $\supp(z) = I$ if and only if $d \mid x$ and 
$$\tfrac{1}{d}(x - \textstyle\sum_{i \in I} a_i) > \mathsf F(\<\tfrac{1}{d}a_i : i \in I\>).$$
As such, for $x \gg 0$, if $x$ has a factorization with support $I$, then $x$ also has a factorization with support $J$ for every $J \supseteq I$.  Thus, $\mathcal L_0(x)$ is an interval and $\Delta_0(x) = \{1\}$.  
\end{proof}

\begin{thm}\label{t:del0families}
The following hold.  
\begin{enumerate}[(a)]
\item 
If $S$ has a minimal presentation $\rho$ in which every trade if between factorizations with singleton support, then $\Delta_0(S) = 1$.  In particular, this occurs whenever $S$ is supersymmetric or generated by a geometric sequence.  

\item 
If $S = \<m, a_1, \ldots, a_{m-1}\>$ with $m \ge 3$ and each $a_i \equiv i \bmod m$ (i.e., $S$ is maximal embedding dimension), then $\Delta_0(S) = \{1,2\}$.  

\item 
If $S = \<a, ah + d, ah + 2d,\ldots, ah + kd\>$ with $h \ge 1$, $2 \le k < a$, and $\gcd(a,d) = 1$ (i.e., $S$ is generated by a generalized arithmetic sequence), then $\Delta_0(S) = \{1,2\}$.  

\end{enumerate}
\end{thm}

\begin{proof}
Part~(a) follows from the fact that any two factorizations of an element $x \in S$ are connected by a chain of factorizations in which successive factorizations $z, z'$ differ by a trade in $\rho$, and thus satisfy $|\ell_0(z) - \ell_0(z')| \le 1$.  As such, $\Delta_0(x) = \{1\}$.  The claims about supersymmetric numerical semigroups and semigroups generated by geometric sequences immediately follow~\cite{supersymmetric,compseqs}.  

For part~(b), by \cite[Theorem~8.30]{numerical} $S$ has a minimal presentation in which each trade has the form
$$
e_i + e_j \sim e_k + c e_0
\qquad \text{with} \qquad
i + j \equiv k \bmod m
\qquad \text{and} \qquad
c \in \ZZ_{\ge 1},
$$
so by similar reasoning to part~(a), $\Delta_0(S) \subseteq \{1,2\}$.  Moreover, since $m \ge 3$, applying the trade with $i = 1$ and $j = 2$ to the factorization $z = e_0 + \cdots + e_{m-1}$ yields a factorization $z'$ with $\ell_0(z') = m-2$.  Moreover, no other factorization $z''$ can have $\ell_0(z'') = m-1$, as then the trade $z \sim z''$ would be between distinct factorizations for a minimal generator of $S$.  

For part~(c), in the minimal presentation for $S$ presented in \cite[Theorem~2.16]{omidalirahmati}, each trade is between factorizations with 0-length at most 2, so $\Delta_0(S) \subseteq\{1,2\}$.  Moreover, writing $a - 1 = qk + r$ with $0 \le r < k$, the minimal presentation in \cite{omidalirahmati} also implies
$$
x = a + (ah + (r+1)d) + q(ah + kd) = a(d + h(q+1))
$$
are the only two factorizations of $x$, so $\Delta_0(x) = \{2\}$.  
\end{proof}

We next characterize $\Delta_0(S)$ when $S$ is 3-generated.  Recall that an expression 
$$
S = t'S' + t''S''
\qquad \text{with} \qquad
S' = \<b_1, \ldots, b_r\>
\qquad \text{and} \qquad
S'' = \<c_1, \ldots, c_{k-r}\>
$$
is called a \emph{gluing} if $t' \in S'' \setminus \{c_1, \ldots, c_{k-r}\}$, $t'' \in S' \setminus \{b_1, \ldots, b_r\}$, and $\gcd(t', t'') = 1$; see~\cite[Chapter~9]{numerical} for more on gluings.  Note~that such an expression for $S$ need not be unique.  In particular, if $S = \<a_1,a_2,a_3\>$, then there can be up to 3 such expressions for $S$ as a gluing, each of the form $S = \<a_i\> + t'S'$ for some $i \in \{1,2,3\}$.  

\begin{thm}\label{t:del0gen3}
Suppose $S = \<a_1, a_2, a_3\>$. If $S$ has at most 1 expression as a gluing, then $\Delta_0(S) = \{1,2\}$.  Otherwise, $\Delta_0(S) = \{1\}$.  
\end{thm}

\begin{proof}
If $S$ has at least 2 distinct expressions $S = \<a_i\> + t'S' = \<a_j\> + t''S''$ as a gluing, then we can write
$$
S = \<t'b_1, t't''b_2, t''b_3\>.
$$
Since $t'b_1 \in S' = \<t'b_2, b_3\>$, there exist $z_2, z_3 \in \ZZ_{\ge 0}$ with $t'b_1 = z_2 t'b_2 + z_3 b_3$, and since $\gcd(t',b_3) = 1$, we must have $t' \mid z_2$.  As such, $b_1 = z_1b_2 + z_2b_3$ and thus $b_1 \in \<b_2,b_3\>$.  By~similar reasoning, we know $b_3 \in \<b_1,b_2\>$.  Assuming $b_1 \le b_3$ without loss of generality, this is only possible if $b_1 = b_3$ or $b_2 \mid b_1$.  In particular, $t'b_1$ has a factorization in $S'$ with singleton support.  As~such, by \cite[Theorem~9.2]{numerical}, $S$ has a minimal presentation within which every factorization has singleton support, so $\Delta_0(S) = \{1\}$ by Theorem~\ref{t:del0families}(a).  

Conversely, suppose $S = \<a_1\> + tS'$ with $S' = \<b_1,b_2\>$ is the only expression of $S$ as a gluing.  Then writing $a_1 = z_1b_1 + z_2b_2$, we cannot have $z_2 = 0$, as otherwise 
$$S = \<z_1b_1, t'b_1, t'b_2\> = \<t'b_2\> + b_1\<z_1,t'\>$$
is a second expression of $S$ as a gluing.  Analogously, $z_1 > 0$.  As such, 
$$x = (t + 1)a_1 = a_1 + z_2 t' b_1 + z_3 t' b_2$$
has $\mathcal L_0(x) = \{1,3\}$, so $\Delta_0(x) = \{2\}$.  

This leaves the case where $S$ cannot be expressed as a gluing.  By \cite[Section~10.3]{numerical}, $S$ has a unique minimal presentation comprised of trades
$$
c_1e_1 \sim r_{12}e_2 + r_{13}e_3,
\qquad
c_2e_2 \sim r_{21}e_1 + r_{23}e_3,
\qquad \text{and} \qquad
c_3e_3 \sim r_{31}e_1 + r_{32}e_2
$$
where 
each $r_{ij} > 0$ and each $c_k = r_{ik} + r_{jk}$ for $\{i,j,k\} = \{1,2,3\}$.  We consider cases.  
\begin{itemize}
\item 
If $r_{ij} = r_{ik} = 1$ for some $i$, then $x = a_1 + a_2 + a_3$ has at least one factorization without full support, and any such factorization must have singleton support, so $\Delta_0(x) = \{2\}$.  

\item 
If $r_{ji} \ge 2$ and $r_{ki} \ge 2$ for some $i$, then 
$$x = (c_i + 1)a_i = a_i + r_{ij} a_j + r_{ik} a_k$$
are the only factorizations of $x$, so $\Delta_0(x) = \{2\}$.  

\item 
If $r_{ji} = r_{ki} = 1$ for some $i$, then either $r_{jk} \ge 2$ and $r_{kj} \ge 2$, in which case
$$x = (c_j + 1)a_j = a_j + r_{ji} a_i + r_{jk} a_k$$
are the only factorizations of $x$ and $\Delta_0(x) = \{2\}$, or $r_{jk} = 1$ or $r_{kj} = 1$, meaning we are in the first case above.  

\item 
In all other cases, after possibly reordering $i$, $j$, and $k$, we have $r_{ij} = r_{jk} = r_{ki} = 1$ while $r_{ji}, r_{kj}, r_{ik} \ge 2$.  In this case,
$$x = (c_j + 1)a_j = a_j + r_{ji} a_i + a_k$$
are the only factorizations of $x$, so $\Delta_0(x) = \{2\}$.  

\end{itemize}
In all cases above, we conclude $\Delta_0(S) = \{1,2\}$.  
\end{proof}

Thus far, all semigroups $S$ presented have $\max \Delta_0(S) \le 2$.  We close by presenting two families of numerical semigroups exhibiting more interesting behavior:\ one demonstrating $\Delta_0(S)$ can be an arbitrarily large interval (Theorem~\ref{t:del0interval}), and another demonstrating $\Delta_0(S) \setminus [1, \max \Delta_0(S)]$ can be arbitrarily large (Theorem~\ref{t:del0gaps}).  

\begin{thm}\label{t:del0interval}
For each $k \ge 2$, there exists a numerical semigroup $S$ such that $\Delta_0(S) = \{1, 2, \ldots, k-1\}$.  
\end{thm}

\begin{proof}
Fix distinct primes $p_1, p_2$ with $p_1, p_2 > k$.  
Let $S_2 = \<p_1, p_2\>$, so $\Delta_0(S_2) = \{1\}$.  
Proceeding inductively, assume $S_{i-1} = \<a_1, \ldots, a_{i-1}\>$ has Betti elements $b_1, \ldots, b_{i-2}$ with $\mathsf Z(b_1) = \{p_2e_1, \,  p_1e_2\}$ and for each $j \ge 2$,
$$
\mathsf Z(b_j) = \{p_{j+1}e_{j+1}, \, (k+1-j)e_1 + e_2 + \cdots + e_j\}
$$
for some prime $p_{j+1}$.  Since each $j \le k$, we have $\Delta_0(b_j + a_{j+1}) = \{j\}$ for each $j$.  Letting 
$$a_i = (k+1-i)a_1 + a_2 + \cdots + a_{i-1},$$
we see (i) the above factoriation of $a_i$ is not preceded (under the component-wise partial order) by any factorizations of $b_1, \ldots, b_{i-2}$ (meaning $a_i$ is uniquely factorable in $S_{i-1}$), and (ii) the above factorization of $a_i$ does not precede a factorization of any $b_j$.  
As~such, choosing a prime $p_i > a_i$, the semigroup 
$$
S_i = p_iS_{i-1} + \<a_i\>
$$
is a gluing, so we have $\Betti(S_i) = \{p_ib_1, \ldots, p_ib_{i-2}, p_ia_i\}$ and 
$$\mathsf Z(p_ia_i) = \{p_ie_i, \, (k+1-i)e_1 + e_2 + \cdots + e_{i-1}\}.$$
This ensures $\Delta_0((p_i+1)a_i) = \{i-1\}$ and $\Delta_0(S_i) = \{1, 2, \ldots, i-1\}$.  Thus, at the conclusion of this process, the semigroup $S_k$ has $\Delta_0(S_k)$ as claimed.  
\end{proof}

\begin{thm}\label{t:del0gaps}
For each $k \ge 16$, there exists a numerical semigroup $S = \<a_1, \ldots, a_{k+1}\>$ with $\Delta_0(S) \cap [\tfrac{7}{8}k, k] = \{k-1,k\}$.  
\end{thm}

\begin{proof}
Let $S_2 = \<2,3\>$.  Next, for each $i = 3, \ldots, k$, let 
$$
S_i = 2S_{i-1} + \<2a_{i-2} + a_{i-1}\>
\qquad \text{where} \qquad
S_{i-1} = \<a_1, \ldots, a_{i-1}\>.
$$
Lastly, let 
$$
S = S_{k+1} = 2S_k + \<a_1 + \cdots + a_k\>
\qquad \text{where} \qquad
S_k = \<a_1, \ldots, a_k\>.$$
As each $S_i$ is easily shown to be a gluing, the trades
$$
2e_2 \sim 3e_1,
\qquad
2e_{k+1} \sim e_1 + \cdots + e_k,
\qquad \text{and} \qquad
2e_i \sim 2e_{i-2} + e_{i-1}
\text{ for } 
3 \le i \le k
$$
form a minimal presentation $\rho$ of $S$.  

In what follows, write $S = \<a_1, \ldots, a_{k+1}\>$.  We see by inspection that
$$
\mathsf Z(3a_{k+1}) = \{3e_{k+1}, \, e_1 + \cdots + e_{k+1}\},
$$
since no other trades in $\rho$ can be performed, 
so in particular $\Delta_0(3a_{k+1}) = \{k\}$ and $\Delta_0(2a_{k+1}) = \{k-1\}$.  We claim every other $x \in S$ with $\Delta_0(x)$ nonempty has $\max \Delta_0(x) \le \tfrac{7}{8}k$.   Indeed, any two factorizations of $x$ can be connected by a sequence of trades in $\rho$, and of such trades, the only one that can yield a change in 0-length of more than 2 is the trade $2e_{k+1} \sim e_1 + \cdots + e_k$.  As such, consider factorizations $z, z' \in \mathsf Z(x)$ of the form 
$$
z = u + e_1 + \cdots + e_k
\qquad \text{and} \qquad
z' = u + 2e_{k+1}
$$
for some $u \in \ZZ_{\ge 0}^k$.  
By way of contradiction, suppose $\ell_0(u) \le \tfrac{1}{8}k$, so that 
$$\ell_0(z) - \ell_0(z') \ge k - \tfrac{1}{8}k = \tfrac{7}{8}k.$$

First, suppose $u_i \ge 1$ for some $i \le \tfrac{1}{2}k$, and fix $j$ maximal with $i + 2j \le k$.  Performing the trade 
$$
2e_i + e_{i+1} + e_{i+3} + \cdots + e_{i+2j-1} \sim 2e_{i+2j}
$$
to $z$ yields a factorization $z''$ in which $j \ge \tfrac{1}{4}k$ entries are strictly smaller than in $z$.  However, since $\ell_0(u) \le \tfrac{1}{8}k$, at least $\tfrac{1}{8}k$ entries of $z''$ must be zero.  As such, 
$$
\ell_0(z) - \ell_0(z'') \le \tfrac{1}{4}k
\qquad \text{and} \qquad
\ell_0(z'') - \ell_0(z') \le \tfrac{7}{8}k.
$$

Next, suppose $u_i \ge 1$ for $\tfrac{1}{2}k < i \le k$.
Performing the trade
\begin{align*}
2e_i + e_{i-1} + e_{i-3} + e_{i-5} + \cdots
&\sim 2e_{i-1} + 2e_{i-3} + 2e_{i-5} + \cdots
\\
&\sim e_{i-2} + 2e_{i-4} + 3e_{i-6} + \cdots
\end{align*}
to $z$ yields a factorization $z''$ in which at least $\tfrac{1}{4}k$ entries are strictly smaller than in $z$.  As in the previous case, we obtain
$$
\ell_0(z) - \ell_0(z'') \le \tfrac{1}{4}k
\qquad \text{and} \qquad
\ell_0(z'') - \ell_0(z') \le \tfrac{7}{8}k.
$$

Lastly, since $x \ne 2a_{k+1}, 3a_{k+1}$, the only remaining case is when $u = ca_{k+1}$ with $c \ge 2$.  In this case, one may performe the trade
\begin{align*}
4e_{k+1}
&\sim 2e_1 + 2e_2 + \cdots + 2e_k
\\
&\sim 3e_{k-1} + 4e_{k-3} + 5e_{k-5} + \cdots
\end{align*}
to obtain a factorization with at least $\tfrac{1}{2}k$ zero entries, which completes the proof.  
\end{proof}

%%%%%%%%%%%%%%%%%%%%%%%%%%%%%%%%%%%%%%%%%%%%%%%%%%%%%%%%%%%%%%%%%%%%%%%%%
%%%%%%%%%%%%%%%%%%%%%%%%%%%%%%%%%%%%%%%%%%%%%%%%%%%%
%%%%%%%%%%%%%%%%%%%%%%%%%%%%%%%%%%%%%%%%%%%%%%%%%%%%%%%%%%%%%%%%%%%%%%%%%

%%%%%%%%%%%%%%%%%%%%%%%%%%%%%%%%%%%%%%%%%%%%%%%%%%%%%%%%%%%%%%%%%%%%%%%%%
\end{document}